\documentclass[12pt]{article}
\usepackage{amssymb}
\usepackage{amsthm}
\usepackage{amsmath}
\usepackage{mathrsfs}
\usepackage{verbatim}
\begin{document}

\title{Rectifiability of Optimal Transportation Plans\footnote{RJM's research is partially supported by NSERC grant 217006-08 and BP is supported in part by an NSERC postgraduate scholarship.  MWW is partially supported by NSF grant 0901644.  The work of BP was completed in partial fulfillment of the requirements for a doctoral degree in mathematics from the University of Toronto.}} 

\author{ROBERT J. McCANN\footnote{Department of Mathematics, University of Toronto, Toronto, Ontario, Canada M5S 2E4 mccann@math.toronto.edu}, BRENDAN PASS\footnote{Department of Mathematics, University of Toronto,  Toronto, Ontario, Canada M5S 2E4 bpass@math.utoronto.ca} AND MICAH WARREN\footnote{Department of Mathematics, Princeton University, Princeton, New Jersey, USA 08544 mww@princeton.edu}}

\maketitle

\begin{abstract}The purpose of this note is to show that the solution to the Kantorovich optimal transportation problem is supported on a Lipschitz manifold, provided the cost is $C^{2}$ with non-singular mixed second derivative.  We use this result to provide a simple proof that solutions to Monge's optimal transportation problem satisfy a change of variables equation almost everywhere.
\end{abstract}

\section{Introduction}

Given Borel probability measures $\mu^{+}$ and $\mu^{-}$ on smooth $n$-dimensional manifolds $M^{+}$ and $M^{-}$ respectively and a cost function $c:M^{+} \times M^{-} \to \mathbf{R}$, the Kantorovich problem is to pair the two measures as efficiently as possible relative to $c$.  A precise formulation requires some notation.  For a measure $\gamma$ on $M^{+} \times M^{-}$, we define the marginals of $\gamma$ to be its push forwards under the canonical projections $\pi^{+}$ and $\pi^{-}$; put another way, the marginals are measures on $M^{+}$ and $M^{-}$ respectively given by the formulae $\pi^{+}_{\#} \gamma(A)=\gamma(A \times M^{-})$ and $\pi^{-}_{\#} \gamma(B)=\gamma(M^{+} \times B)$ for all Borel sets $A \subset M^{+}, B \subset M^{-}$.  The Kantorovich problem is then to minimize the functional 
\begin{equation}
\int_{M^{+} \times M^{-}} c(x,y)d\gamma(x,y) 
\end{equation}
among all measures $\gamma$ on $M^{+} \times M^{-}$ whose marginals are $\pi^{+}_{\#} \gamma=\mu^{+}$ and $\pi^{-}_{\#} \gamma=\mu^{-}$.

Under fairly weak conditions, it is straightforward to show that a solution to this problem exists.  In this paper, we study what can be said about that solution under a certain non-degeneracy condition on the cost function, which was originally introduced in an economic context by McAfee and McMillan \cite{McAMcM} and later rediscovered by Ma, Trudinger and Wang \cite{mtw}; in the terminology of Ma, Trudinger and Wang, it is also known as the (A2) condition. 

In what follows, $D^{2}_{xy}c(x_0,y_0)$ will denote the $n$ by $n$ matrix of mixed second order partial derivatives of the function $c$ at the point $(x_0, y_0) \in M^{+} \times M^{-}$; its $(i,j)$th entry is $\frac{d^2c}{dx^idy^j}(x_0,y_0)$. 
\newtheorem{nondegeneracy}{Definition}[section]
\begin{nondegeneracy}

Assume $c \in C^{2}(M^{+} \times M^{-})$.  We say that $c$ is non-degenerate at a point $(x_0, y_0) \in M^{+} \times M^{-}$ if $D^{2}_{xy}c(x_0,y_0)$ is nonsingular; that is if det($D^{2}_{xy}c(x_0,y_0)) \neq 0$.
\end{nondegeneracy}

For a probability measure $\gamma$ on $M^{+} \times M^{-}$ we will denote by spt($\gamma$) the support of $\gamma$; that is, the smallest closed set $S \subseteq M^{+} \times M^{-}$ such that $\gamma(S)=1$.

Our main result is:

\newtheorem{rectifiable}[nondegeneracy]{Theorem} 
\begin{rectifiable}\label{rectifiable}

Suppose $c \in C^{2}(M^{+} \times M^{-})$ and $\mu^+$ and $\mu^-$ are compactly supported; let $\gamma$ be a solution of the Kantorovich problem.  Suppose $(x_0,y_0) \in spt(\gamma)$ and $c$ is non-degenerate at $(x_0,y_0)$.  Then there is a neighbourhood $N$ of $(x_0,y_0)$ such that $N \cap spt(\gamma)$ is contained in an $n$-dimensional Lipschitz submanifold.  In particular, if $D^{2}_{xy}c$ is nonsingular everywhere, $spt\gamma$ is contained in an $n$-dimensional Lipschitz submanifold.
\end{rectifiable}
The proof of this theorem is based on an idea of Minty \cite{minty},  which was also used by Alberti and Ambrosio to show that the graph of any monotone function $T: \textbf{R}^n \rightarrow \textbf{R}^n$ is contained in a Lipschitz graph over the diagonal $\Delta=\{u=\frac{x+y}{\sqrt{2}}:(x,y) \in \textbf{R}^n \times \textbf{R}^n\}$ \cite{aa}.

The non-degeneracy condition can be viewed as a linearized version of the twist condition, which asserts that the mapping $y \in M^{-} \longmapsto D_xc(x,y)$ is injective.  Under suitable regularity conditions on the marginals, Levin \cite{lev} showed that the twist condition ensures that the solution to the Kantorovich problem is concentrated on the graph of a function and is therefore unique; see also Gangbo \cite{g}.  For the past two decades, the regularity of these maps has been an active area of investigation.  Regularity results were proven for the quadratic cost function by Caffarelli \cite{c1}\cite{c2}\cite{c3}, Delano\"e \cite{d1}\cite{d2} and Urbas \cite{u} and for another special cost function by Wang \cite{wang}.  These were then generalized by Ma, Trudinger and Wang \cite{mtw}, who discovered a fourth order differential condition on the cost function that ensures the optimal map is smooth, provided the marginals are sufficiently regular \cite{mtw}\cite{tw1}.  Our results assert that something can be said about the smoothness of the support even without these strong conditions on the cost and the marginals, provided that one is willing to view the support as a submanifold rather than a graph.  

In one dimension, non-degeneracy implies twistedness, as was noted by many authors, including Spence \cite{s} and Mirrlees \cite{mir}, in the economics literature; see also \cite{m2}.  In higher dimensions, this is no longer true; the non-degeneracy condition will imply that the map $y \in M^{-} \longmapsto D_xc(x,y)$ is locally injective but not necessarily globally. Non-degeneracy was a hypothesis in the smoothness proof in \cite{mtw}, but does not seem to have received much attention in higher dimensions before then.  While our result demonstrates that the non-degeneracy condition is enough to ensure that solutions still have certain regularity properties, we will show by example that the uniqueness result that follows from twistedness can fail for non-degenerate costs which are not twisted.  The twist condition is asymmetric in $x$ and $y$; that is, there are cost functions for which the map $y \in M^{-} \longmapsto D_xc(x,y)$ is injective but $x \in M^{+} \longmapsto D_y
 c(x,y)$ is not.  However, since $(D^2_{xy}c)^T=D^2_{yx}c$ the non-degeneracy condition is certainly symmetric in $x$ and $y$.  In view of this, it is not surprising that the twist condition can only be used to show solutions are concentrated on the graphs of functions of $y$ over $x$ whereas the non-degeneracy condition implies solution are concentrated on $n$-dimensional submanifolds, a result that does not favour either variable over the other.

Smooth optimal maps solve certain Monge-Amp\`ere type equations.  Typically, an optimal map will be differentiable almost everywhere, but may not be smooth.  It has proven useful to know when non-smooth optimal maps solve the corresponding equations almost everywhere.  Formally, the link between optimal transportation and these equations was observed by Brenier \cite{bren}, then Gangbo and McCann \cite{gm}, and they were studied in detail by Ma, Trudinger and Wang \cite{mtw}.  An important step in showing that an optimal map solves a Monge-Amp\`ere type equation is first showing that it solves the Jacobian --- or change of variables --- equation.  An injective Lipschitz function satisfies the change of variables formula almost everywhere, so some sort of Lipschitz rectifiability for the graphs of optimal maps is a useful tool in resolving this question. As an application of Theorem 1.2, we provide a simple proof that optimal maps satisfy the change of variables formula almost everywhere.

This work is related to another interesting line of research.  A measure $\gamma$ on the product $M^{+} \times M^{-}$ is called {\em simplicial} if it is extremal among the convex set of all measures which share its marginals.  There are a number of results describing simplicial measures and their supports \cite{d}\cite{l}\cite{bs}\cite{hw}\cite{akm}.  One consequence is that the support of simplicial measures are in some sense small; in particular, the support of a simplicial measure on $[0,1] \times [0,1]$  must have two-dimensional Lebesgue measure zero \cite{l}\cite{hw}.  However, any measure supported on the graph of a function is simplicial and it is known that there exist functions whose graphs have Hausdorff measure $2-\epsilon$, for any $\epsilon > 0$ \cite{adl}.    For any cost, the Kantorovich functional is linear and is hence minimized by some simplicial measure.  Conversely, any simplicial measure is the solution to a Kantorovich problem for some continuous cost function, and so by the remarks above there are continuous cost functions whose optimizers are supported on sets of Hausdorff dimension  $2-\epsilon$.   On the other hand, an immediate consequence of our result is that the support of optimizers of Kantorovich problems with non-degenerate $C^2$ costs have Hausdorff dimension at most $n$.

The result of Ma, Trudinger and Wang proving smoothness of the optimal map under certain conditions immediately implies that the support of the optimizer has Hausdorff dimension $n$; however, the proof of this result requires that the marginals be $C^2$ smooth.  Under the same assumptions on the cost functions but weaker regularity conditions on the marginals, Loeper \cite{loeper} and Liu \cite{liu} have demonstrated that the optimal map is H\"older continuous for some H\"older constant $0<\alpha<1.$  It is worth noting that there are examples of functions on $\mathbf{R}^{n}$ \cite{adl} which are H\"older continuous with exponent $\alpha$ but whose graphs have Hausdorff dimension $n+1-\alpha$, so the latter results do not imply that the Hausdorff dimension of the optimizer must be $n$.

In the second section of this manuscript we prove Theorem 1.2 while Section 3 is devoted to discussion and examples. In the final section, we use Theorem 1.2 to provide a simple proof that optimal maps satisfy a prescribed Jacobian equation almost everywhere.

We are pleased to acknowledge that our interest in this topic was stimulated in part by a fruitful discussion between one of the authors and Ivar Ekeland.

\section{Lipschitz Rectifiability of Optimal Transportation Plans}

We now prove Theorem 1.2.  Note that $\gamma$ minimizes the Kantorovich functional if and only if it maximizes the corresponding functional for $b(x,y) = -c(x,y)$.  To simplify the computation, we consider $\gamma$ that maximizes $b$.

Our proof relies on the \textit{b-monotonicity} of the supports of optimal measures:

\newtheorem{mono}{Definition}[section]
\begin{mono}
 A subset $S$ of $M^{+} \times M^{-}$ is $b$-monotone if all $(x_0,y_0), (x_1,y_1) \in S$ satisfy $b(x_0,y_0) + b(x_1,y_1) \geq b(x_0,y_1) + b(x_1,y_0)$.
\end{mono}

It is well known that the support of any optimizer is $b$-monotone \cite{sk}, provided that the cost is continuous and the marginals are compactly supported.  The reason for this is intuitively clear; if $b(x_0,y_0) +b(x_1,y_1) > b(x_0,y_1) + b(x_1,y_0)$ then we could move some mass from $(x_0, y_0)$ and $(x_1,y_1)$ to $(x_0,y_1)$ and $(x_1,y_0)$ without changing the marginals of $\gamma$ and thus increase the integral of $b$.

The strategy of our proof is to change coordinates so that locally $b(x,y)=x\cdot y$, modulo a small perturbation.  We then switch to diagonal coordinates $u=x+y, v=x-y$ and show that the monotonicity condition becomes a Lipschitz condition for $v$ as a function of $u$.  This trick dates back to Minty who used it to study monotone operators on Hilbert spaces \cite{minty}; more recently, Alberti and Ambrosio used it to investigate the fine properties of monotone functions on $\textbf{R}^{n}$ \cite{aa}.

We are now ready to prove Theorem 1.2:

\begin{proof}

Choose $(x_{0}, y_{0})$ in the support of $\gamma$.  Changing coordinates in a neighbourhood
of $y_0$ yields $D^{2}_{xy} b(x_{0}, y_{0}) = I$ without loss of generality.
We then have
$b(x,y)=x\cdot y + G(x,y)$, where  $D^{2}_{xy}G \rightarrow 0$ as
$(x,y) \rightarrow (x_{0},y_{0})$.  Set
$u\sqrt{2}=x+y$ and $v\sqrt{2}=y-x.$
Given $\epsilon>0$, choose a convex neighbourhood $N$ of $(x_{0},y_{0})$ such that
$||D^{2}_{xy}G|| \leq \epsilon$ on $N$.  We will show that
$\gamma \cap N$ is contained in a Lipschitz graph of $v$ over $u$;
hence, $u$ and $v$ serve as local coordinates for our submanifold.
Take $(x,y)$ and $(x',y') \in N \cap spt \gamma$. Then, by
$b$-monotonicity, we have $b(x,y) +b(x',y') \geq b(x,y') + b(x',y)$, hence
\begin{eqnarray}
x\cdot y +G(x,y)+ x'\cdot y' +G(x',y')
\nonumber \\ \geq
x\cdot y' +G(x,y')+ x'\cdot y +G(x',y).
\nonumber \end{eqnarray}
Setting $\Delta x = x'-x$,  $\Delta y = y'-y$, $\Delta u= u'-u$, $\Delta v = v' - v$,
and rewriting yields
\begin{eqnarray}
(\Delta x)\cdot (\Delta y) + (\Delta x) \cdot \int_{0}^{1}\int_{0}^{1} D^{2}_{xy}G[x+s\Delta x,y+t\Delta y](\Delta y)dsdt
\geq 0 \end{eqnarray}
which simplifies to:
$\Delta x \cdot \Delta y \geq -\epsilon |\Delta x| |\Delta y|$.

Observe that $\Delta y\sqrt{2}=\Delta u + \Delta v$ and
$\Delta x \sqrt{2}=\Delta u - \Delta v$.
Now, \begin{eqnarray}
|\Delta u|^{2} - |\Delta v|^{2}
&=& 2(\Delta x)\cdot (\Delta y) \nonumber \\
&\geq& -2\epsilon|\Delta x||\Delta y| \nonumber \\
&=& -\epsilon|\Delta u - \Delta v| |\Delta u + \Delta v| \nonumber \\
&\geq& -\epsilon[|\Delta u|^{2}+|\Delta v|^{2}] \nonumber
\end{eqnarray}
The last inequality follows by squaring the absolute values of each side and expanding the first term.  Rearranging yields $(1+\epsilon)|\Delta u|^{2} \geq (1-\epsilon) |\Delta v|^{2}$, the desired result.

Note that $v$ may not be everywhere defined; that is, for certain values of $u$ there may be no corresponding $v$ in $spt(\gamma)$.   However, the function $v(u)$ can be extended by Kirzbraun's theorem and hence we can conclude that $spt(\gamma)$ is contained in the graph of a Lipschitz function of $v$ over $u$.

\end{proof}

\newtheorem{rem}{Remark}[section]
\begin{rem}
Note that the only property of optimal transportation plans used in the proof is $b$-monotonicity, so we have actually proven that any $b$-monotone subset of $M^+ \times M^-$ is contained in an $n$-dimensional Lipschitz submanifold, provided $b$ is non-degenerate. 
\end{rem} 
\section{Discussion and Examples}
For twisted costs, one can show that $spt(\gamma)$ is concentrated on the graph of a function, provided the marginal $\mu^{+}$ does not charge sets whose dimension is less than or equal to $n-1$\cite{g} \cite{lev} \cite{mtw} \cite{akm} \cite{m3} \cite{gm}\footnote{In fact, this condition on the regularity of $\mu^+$ has recently been sharpened \cite{gig}.}; however, this can fail if $\mu^{+}$ charges small sets. On the other hand, notice that our proof did not require any regularity hypotheses on the marginals.

In the example below, we exhibit a non-degenerate cost which is not twisted.  We use this example to illustrate how, in this setting, solutions may be supported on submanifolds which are are not necessarily graphs.  In addition, we show that these solutions may not be unique.  We can view this example as expressing an optimal transportation problem on a right circular cylinder via its universal cover, which is $\mathbf{R}^2$.  The non-twistedness of the cost and non-uniqueness of the solution arise because different points in the universal cover correspond to the same point in the cylinder and are therefore indistinguishable by our cost function.  In fact, if we expressed the problem on the cylinder, we would have a twisted cost function and therefore a unique solution.

\newtheorem{ex}{Example}[section]
\begin{ex}
Let $M^{\pm}=\mathbf{R}^2$ and $c(x,y)=e^{x_1+y_1}$cos$(x_2-y_2)+\frac{e^{2x_1}}{2}+\frac{e^{2y_1}}{2}$.  Then $D_xc(x,y)=(e^{x_1+y_1}$cos$(x_2-y_2)+{e^{2x_1}}, -e^{x_1+y_1}$sin$(x_2-y_2))$, so $y \in M^{-} \longmapsto D_xc(x,y)$ is not injective and $c$ is not twisted.  However, note that 
$D^{2}_{xy}c(x,y)=$
\begin{equation*} \qquad
\begin{bmatrix}
 e^{x_1+y_1}cos(x_2-y_2) & e^{x_1+y_1}sin(x_2-y_2) \\
-e^{x_1+y_1}sin(x_2-y_2) & e^{x_1+y_1}cos(x_2-y_2) \\
\end{bmatrix}
\end{equation*}

Therefore, det$D^{2}_{xy}c(x,y)= e^{2(x_1+y_1)} >0$ for all $(x,y)$, so $c$ is non-degenerate.  Optimal measures for $c$, then, must be supported on $2$-dimensional Lipschitz submanifolds, but we will now exhibit an optimal measure whose support is not contained in the graph of a function.

Now let $M$ be the union of the three graphs: 
\begin{eqnarray}
G_1: y_1=x_1, y_2 = x_2+\pi \\
G_2: y_1=x_1, y_2 = x_2+3\pi \\
G_3: y_1=x_1, y_2 = x_2+5\pi 
\end{eqnarray}
Clearly, $M$ is a smooth $2$-$d$ submanifold but not a graph.  However, $c(x,y) \geq -e^{x_1+y_1}+\frac{e^{2x_1}}{2}+\frac{e^{2y_1}}{2} \geq \frac{(e^{x_1}-e^{y_1})^{2}}{2}$ and we have equality on $M$.  Therefore, any probability measure whose support is concentrated on $M$ is optimal for its marginals.

We now show that optimal measures supported on $M$ may not be unique.  Let $ S=\{((x_1,x_2) , (y_1,y_2)) | 0 \leq x_1 \leq 1, 0 \leq x_2 \leq 4\pi \}$. Note that 
\begin{equation*}
M \cap S = (G_1 \cap S) \cup (G_2 \cap S) \cup(G_3 \cap S).  
\end{equation*}
consists of 3, flat $2$-$d$ regions.  Let $\gamma$ be uniform measure on these regions.  Now, let $\overline{\gamma_1}$ be uniform measure on the the first half of $G_1 \cap S$; that is, on
\begin{equation*}
 G_1 \cap \{((x_1,x_2) , (y_1,y_2)) | 0 \leq x_1 \leq 1, 0 \leq x_2 \leq 2\pi \}.  
\end{equation*} 
Let $\overline{\gamma_3}$ be uniform measure on the the second half of $G_3 \cap S$, or 
\begin{equation*}
G_3 \cap \{((x_1,x_2) , (y_1,y_2)) | 0 \leq x_1 \leq 1, 2\pi \leq x_2 \leq 4\pi \}.
\end{equation*} 
Take $\overline{\gamma_2}$ to be twice uniform measure on $G_2 \cap S$ and set $\overline{\gamma}=\overline{\gamma_1}$+$\overline{\gamma_2}$+$\overline{\gamma_3}$.  Then $\gamma$ and $\overline{\gamma}$ share the same marginals and are both optimal measures.  Furthermore, any convex combination  $t\gamma+(1-t)\overline{\gamma}$ will also share the same marginals and will be optimal as well.

\end{ex}

The next example is similar in that the cost function is non-degenerate but not twisted.  However, this  cost {\em{would}} be twisted if we exchanged the roles of $x$ and $y$.  This demonstrates that, unlike non-degeneracy, the twist condition is not symmetric in $x$ and $y$.  For this cost function, solutions will be unique as long as the {\em{second}} marginal does not charge small sets.   

\newtheorem{ex2}[ex]{Example}
\begin{ex2}
Let $M^{\pm}=\mathbf{R}^2$ and $c(x,y)=-(x_1$cos$(y_1)+x_2$sin$(y_1))e^{y_2}+\frac{e^{2y_2}}{2}+\frac{x_1^2+x_2^2}{2}$.  Note that det$D^{2}_{xy}c(x,y)= -e^{2y_2} <0$, so $c$ is non-degenerate.  However, $D_xc(x,y)=(-$cos$(y_1)e^{y_2}+x_1, -$sin$(y_1)e^{y_2}+x_2)$, so $y \in M^{-} \longmapsto D_xc(x,y)$ is not injective and $c$ is not twisted.  On the other hand, $D_yc(x,y)=((x_1$sin$(y_1)+x_2$cos$(y_1))e^{y_2},-(x_1$cos$(y_1)+x_2$sin$(y_1))e^{y_2}+e^{2y_2})$ and so $x \in M^{+} \longmapsto D_yc(x,y)$ is injective.  This implies that solutions are supported on graphs of $x$ over $y$ but that these graphs are not necessarily invertible.  In fact, $c(x,y)\geq \frac{((x_1^2+x_2^2)^{\frac{1}{2}}-e^{y_2})^2}{2} \geq 0$, where equality holds if and only if cos$(y_1)=\frac{x_1}{(x_1^2+x_2^2)^{\frac{1}{2}}}$, sin$(y_1)=\frac{x_2}{(x_1^2+x_2^2)^{\frac{1}{2}}}$, and $(x_1^2+x_2^2)^{\frac{1}{2}}=e^{y_2}$.  This set of equality is a non-invertible graph of $x$ over $y$; any measure whose
  support is contained in this graph is optimal for its marginals.  Note that as any minimizer for this problem must be supported on this graph, the solution is unique \cite{akm}.
\end{ex2}

\newtheorem{remark}[ex]{Remark}
\begin{remark}

For twisted costs with regular marginals, any solution is concentrated on the graph of a particular function \cite{mtw}.  It is not hard to show that at most one measure with prescribed marginals can be supported on such a graph; hence, uniqueness of the optimizer follows immediately. 

While our result asserts that for non-degenerate costs the solution concentrates on some $n$-dimensional Lipschitz submanifold, the proof says little more about the submanifold itself.  In contrast to the twisted setting, then, our result cannot be used to deduce a uniqueness argument.  Furthermore, as the example above shows, even if we do know the support of the optimizer explicitly, solutions may not be unique if this support is not concentrated on the graph of a function.

\end{remark}

Theorem 1.2 also says something about problems where $D^{2}_{xy} c$ is allowed to be singular, but where the gradient of its determinant is non-zero at the singular points.  In this case, the implicit function theorem implies that the set where $D^{2}_{xy} c$ is singular has Hausdorff dimension $2n-1$.  Theorem 1.2 is valid wherever $D^{2}_{xy} c$ is nonsingular, so that the optimal measure is concentrated on the union of a smooth $2n-1$ dimensional set and an $n$ dimensional Lipschitz submanifold.  For example, when $n=1$, this shows that the support of the optimal measure is $1$ dimensional.

\section{A Jacobian equation}

We now provide a simple proof that an optimal map satisfies a prescribed Jacobian equation almost everywhere.  This result was originally proven for the quadratic cost in $\mathbf{R}^{n}$ by McCann \cite{m}, and for the quadratic cost on a Riemannian manifold by Cordero-Erasquin, McCann and Schmuckenschl\"ager \cite{c-ems}.  Cordero-Erasquin generalized this approach to deal with strictly convex costs on $\mathbf{R}^{n}$ \cite{c-e}; see also \cite{agu}.
It was observed by Ambrosio, Gigli and Savare that this can be deduced from results in \cite{afp} and \cite{ags} when the optimal map is approximately differentiable, which is true even for some non-smooth costs.  Our method works only when the cost is $C^{2}$ and non-degenerate, but has the advantage of a simpler proof, relying only on the area/coarea formula for Lipschitz functions.

For a Jacobian equation to make sense, the solution must be concentrated on the graph of a function, and that function must be differentiable in some sense, at least almost everywhere.  A twisted cost suffices to ensure the first condition.  The second follows from the smoothness and non-degeneracy of $c$.  Recall that for a twisted cost the optimal map has the form $T(x)=c$-$exp_x(Du(x))$; as $c$-$exp_x(\cdot)$ is the inverse of $y\longmapsto D_xc(x,y)$, its differentabiliy follows from the non-degeneracy of $c$ and the inverse function theorem.  The almost everywhere differentiability of $Du(x)$ (or, equivalently, the almost everywhere twice differentiability of $u$) follows from  $C^{2}$ smoothness of $c$; $u$ takes the form $u(x) = $inf$_y(c(x,y)-v(y))$ for some function $v(y)$ and is hence semi-concave \cite{gm}.  In the present context, we need only the weaker condition that the optimal map is continuous almost everywhere; its differentiability will follow from Theorem
  \ref{rectifiable}.

\newtheorem{Jacobi}{Proposition}[section]
\begin{Jacobi}

Assume that the cost is non-degenerate and that an optimizer $\gamma$ is supported on the graph of some function $T: dom(T) \to M^{-}$ which is injective and continuous when restricted to a set $dom(T) \subseteq M^{+}$ of full Lebesgue measure.  Suppose that the marginals are absolutely continuous with respect to volume; set $d\mu^{+} = f^{+}(x)dx$ and $d\mu^{-} = f^{-}(y)dy$.  Then, for almost every $x$, $f^{+}(x) = |$det$DT(x)|f^{-}(T(x))$.

\end{Jacobi}

\begin{proof}
Choose a point $x$ where $T$ is continuous and a neighbourhood $U^{-}$ of $T(x)$ such that for $U^{+}=T^{-1}(U^{-})$, the part of the optimal graph contained in $U^{+} \times U^{-}$ lies in a Lipschitz graph $v=G(u)$ over the diagonal $\Delta=\{u=\frac{x+y}{\sqrt{2}}:(x,y) \in U^{+} \times U^{-}\}$, after a change of coordinates.  Now $x=\frac{u+v}{\sqrt{2}}$ and $y=\frac{u-v}{\sqrt{2}}$, so the optimal measure is supported on the graph of the Lipschitz function $(x,y)=(F^{+}(u),F^{-}(u)):=(\frac{u+G(u)}{\sqrt{2}}, \frac{u-G(u)}{\sqrt{2}})$. By projecting onto the diagonal, we obtain a measure $\nu$ on $\Delta$ that pushes forward to $\mu^{+}|_{U^{+}}$ and $\mu^{-}|_{U^{-}}$ under the Lipschitz mappings $F^{+}$ and $F^{-}$, respectively.  Now, as $F^{+}$ is Lipschitz, the image of any zero volume set must also have zero volume; as $\mu^{+}|_{U^{+}}$ is absolutely continuous with respect to Lebesgue, $\nu$ must be as well; we will write $\nu =
  h(u)du$.   
Now, for almost every $x \in U^{+}$ there is a unique $y=T(x)$ such that $(x,y) \in \text{spt}(\gamma)$ and hence a unique $u=\frac{x+y}{\sqrt{2}}$ on the diagonal such that $x=F^{+}(u)$.  It follows that the map $F^{+}$ is one to one almost everywhere and so for every set $A \subseteq \Delta$ we have $\int_{A}h(u)du=\int_{F^{+}(A)}f^{+}(x)dx$.  But the right hand side is $\int_{A}f^{+}(F^{+}(u))|$det$DF^{+}(u)|du$ by the area formula; as $A$ was arbitrary, this means $h(u)=f^{+}(F^{+}(u))|$det$DF^{+}(u)|$ almost everywhere.  Similarly, $h(u)=f^{-}(F^{-}(u))|$det$DF^{-}(u)|$ almost everywhere,  hence 
\begin{equation*}
f^+(F^{+}(u))|\text{det}DF^{+}(u)|=f^{-}(F^{-}(u))|\text{det}DF^{-}(u)|
\end{equation*}
almost everywhere.  As the image under $F^{+}$ of a negligible set must itself be negligible, we have  
\begin{equation}
f^{+}(x)|\text{det}DF^{+}((F^{+})^{-1}(x))|=f^{-}(F^{-}((F^{+})^{-1}(x)))|\text{det}DF^{-}((F^{+})^{-1}(x))|
\end{equation}
for almost all $x$. Note that as $F^{+}$ is one to one almost everywhere and $F^{+}(\{u \in \Delta: $det$DF^{+}(u)=0\})$ has measure zero by the area formula, $(F^{+})^{-1}$ is differentiable almost everywhere.  As $T \circ F^{+} =F^{-}$, it follows that $T$ is differentiable almost everywhere and  
 \begin{equation*}
 \text{det}DT(F^{+}(u))\text{det}DF^{+}(u)=\text{det}DF^{-}(u) 
 \end{equation*}
 whenever $F^{+}$ and $F^{-}$ are differentiable at $u$ and $T$ is differentiable at $F^{+}(u)$.  Hence, 
  \begin{equation}
 \text{det}DT(x)\text{det}DF^{+}((F^{+})^{-1}(x))=\text{det}DF^{-}((F^{+})^{-1}(x)) 
  \end{equation}
 for all $x$ such that $T$ is differentiable at $x$ and $F^{+}$ and $F^{-}$ are differentiable at $(F^{+})^{-1}(x)$.  $T$ is differentiable for almost every $x$ , $F^{+}$ and $F^{-}$ are differentiable for almost every $u$ and $F^{+}$ is Lipschitz; it follows that the above holds almost everywhere.  Now, combining $(6)$ and $(7)$ we obtain $f^{+}(x) = |$det$DT(x)|f^{-}(T(x))$ for almost every $x$.

\end{proof}
\newtheorem{diff}{Remark}[section]
\begin{diff}
Note that the preceding proposition does not require that continuity of $T$ extend outside $dom(T)$.  Thus it applies to $T= Du$, for example, where $u$ is an arbitrary convex function and $dom(T)$ is its domain of differentiability.
\end{diff}

\end{document}